\documentclass[a4paper,reqno]{amsart} 

\usepackage{amssymb}
\usepackage[mathcal]{euscript} 

\usepackage{tikz-cd} 
\usetikzlibrary{babel}

\usepackage{hyperref}

\newcommand{\bydef}{:=}

\newcommand{\id}{\mathrm{id}}

\newcommand{\cA}{\mathcal{A}}
\newcommand{\cB}{\mathcal{B}}

\newcommand{\cH}{\mathcal{H}} 
\newcommand{\cI}{\mathcal{I}}

\newcommand{\cM}{\mathcal{M}}

\newcommand{\cV}{\mathcal{V}}

\newcommand{\NN}{\mathbb{N}}

\newcommand{\QQ}{\mathbb{Q}}

\newcommand{\FF}{\mathbb{F}} 

\DeclareMathOperator{\Mod}{\mathrm{Mod}}
\DeclareMathOperator{\Grp}{\mathrm{Grp}}
\DeclareMathOperator{\AlgF}{\mathrm{Alg_{\FF}}}
\DeclareMathOperator{\AlgR}{\mathrm{Alg_{R}}}
\DeclareMathOperator{\Hom}{\mathrm{Hom}}

\DeclareMathOperator{\Alg}{\mathrm{Alg}}

\DeclareMathOperator{\Aut}{\mathrm{Aut}}

\DeclareMathOperator{\Stab}{\mathrm{Stab}}
\DeclareMathOperator{\Diag}{\mathrm{Diag}}

\DeclareMathOperator{\supp}{\mathrm{Supp}}

\DeclareMathOperator{\Cent}{\mathrm{Cent}}
\DeclareMathOperator{\Norm}{\mathrm{Norm}}

\DeclareMathOperator{\Sym}{\mathrm{Sym}}

\newcommand{\GL}{\mathrm{GL}}

\newcommand{\Auts}{\mathbf{Aut}}
\newcommand{\Cents}{\mathbf{Cent}}
\newcommand{\Norms}{\mathbf{Norm}}

\newcommand{\Gs}{\mathbf{G}}

\newcommand{\Hs}{\mathbf{H}}

\newcommand{\Ns}{\mathbf{N}}
\newcommand{\Ds}{\mathbf{D}}
\newcommand{\GLs}{\mathbf{GL}}

\newcommand{\Diags}{\mathbf{Diag}}
\newcommand{\bmu}{\boldsymbol{\mu}}
\newcommand{\Stabs}{\mathbf{Stab}}

\newcommand{\Csf}{\mathsf{C}}
\newcommand{\Wsf}{\mathsf{W}}
\newcommand{\Symsf}{\mathsf{Sym}}

\newtheorem{theorem}{Theorem}[section]
\newtheorem{proposition}[theorem]{Proposition}

\newtheorem{corollary}[theorem]{Corollary}

\theoremstyle{definition} 
\newtheorem{definition}[theorem]{Definition}
\newtheorem{example}[theorem]{Example}

\theoremstyle{remark} \newtheorem{remark}[theorem]{Remark}
\numberwithin{equation}{section}

\begin{document}

\title
{Automorphism group schemes and Weyl groups of gradings}

\author[A.~Elduque]{Alberto Elduque} 
\address{Departamento de
Matem\'{a}ticas e Instituto Universitario de Matem\'aticas y
Aplicaciones, Universidad de Zaragoza, 50009 Zaragoza, Spain}
\email{elduque@unizar.es} 
\thanks{Supported by grant
PID2021-123461NB-C21, funded by 
MCIN/AEI/10.13039/ 501100011033 and by
 ``ERDF A way of making Europe''; and by grant 
E22\_20R (Gobierno de Arag\'on).}




\begin{abstract}
A scheme theoretic version of the automorphism group of a grading on an algebra is presented, and 
the classical result that shows that, over algebraically closed fields of characteristic $0$, 
the automorphism group
of a grading is the normalizer of its diagonal group is extended, over arbitrary fields,
to this scheme setting.

The quotient of the scheme theoretic versions of the automorphism group and the stabilizer of a grading
turns out to be a constant group scheme, called the Weyl group scheme of the grading. For algebraically
closed fields this is the constant group scheme associated to the ordinary Weyl group of the grading,
but this fails over arbitrary fields.
\end{abstract}

\maketitle

\begin{center}
\textit{Dedicated to professor Leonid Makar-Limanov, on the occasion of his 80th birthday}
\end{center}

\bigskip

\section{Introduction}

Let $\cA$ be a nonassociative (i.e., not necessarily associative) algebra over a field $\FF$ and let $G$ be an abelian group. A \emph{grading} on $\cA$ by $G$, or \emph{$G$-grading}, is a vector space decomposition
\begin{equation}\label{eq:Gamma}
\Gamma: \cA=\bigoplus_{g\in G}\cA_g,
\end{equation}
satisfying $\cA_g\cA_h\subseteq \cA_{gh}$ for all $g,h\in G$. In this case, the nonzero elements in 
$\cA_g$ are said to be \emph{homogeneous of degree} $g$ and we write $\deg x=g$ for 
$0\neq x\in\cA_g$. 
The subspace $\cA_g$ is the \emph{homogeneous component} of degree $g$.

The set
\[
\supp(\Gamma)\bydef \{g\in G \mid \cA_g\ne 0\}
\]
is called the \emph{support} of $\Gamma$.

There are some natural subgroups of the group $\Aut(\cA)$ of automorphisms of $\cA$  attached to 
$\Gamma$ (see \cite{PZ89} and \cite[Chapter 1]{EKmon}). The first one is the \emph{stabilizer of the grading}:
\[
\Stab(\Gamma)\bydef\{\varphi\in\Aut(\cA)\mid \varphi(\cA_g)\subseteq\cA_g\ \forall g\in G\}.
\]
Another one is the \emph{automorphism group of the grading}:
\[
\Aut(\Gamma)\bydef\{\varphi\in\Aut(\cA)\mid \forall g\in G\ \exists h\in G\text{ such that }
\varphi(\cA_g)\subseteq \cA_h\}.
\]
Each $\varphi\in\Aut(\Gamma)$ induces a permutation of $\supp(\Gamma)$, and this gives an exact 
sequence
\[
\begin{tikzcd}
1\arrow[r, rightarrow] & \Stab(\Gamma) \arrow[r, hookrightarrow] 
&\Aut(\Gamma) \arrow[r, rightarrow]& \Sym\bigl(\supp(\Gamma)\bigr).
\end{tikzcd}
\]
The quotient
\[
W(\Gamma)\bydef \Aut(\Gamma)/\Stab(\Gamma)
\]
is called the \emph{Weyl group} of the grading and is (isomorphic to) a subgroup of the symmetric 
group $\Sym\bigl(\supp(\Gamma)\bigr)$. The Weyl group measures how much symmetric a grading is.

\smallskip

Assume for a while that the ground field $\FF$ is algebraically closed of characteristic $0$, 
and that $\cA$ is finite dimensional. Assume too, without loss of generality, that $G$ is finitely
generated, otherwise substitute $G$ by the subgroup generated by $\supp(\Gamma)$.
Consider the group of characters $\widehat{G}=\Hom(G,\FF^\times)$
(group homomorphisms). The grading $\Gamma$ induces a homomorphism (of algebraic groups)
\begin{equation}\label{eq:rhoG}
\eta_\Gamma\colon \widehat{G}\longrightarrow \Aut(\cA),
\end{equation}
that takes any character $\chi\in\widehat{G}$ to the automorphism $\tau_\chi\in\Aut(\cA)$ given by:
\begin{equation}\label{eq:tauchi}
\tau_\chi\colon x\in \cA_g\longmapsto \chi(g)x\quad \forall g\in G.
\end{equation}
Since characters separate elements of $G$, the homogeneous components of $\Gamma$ can be
retrieved from $\eta_\Gamma$: for all $g\in G$ we have
\[
\cA_g=\{x\in\cA\mid \tau_\chi(x)=\chi(g)x\ \forall \chi\in\widehat{G}\}.
\]

The image of $\eta_\Gamma$ is contained in the \emph{diagonal group} of $\Gamma$:
\begin{equation}\label{eq:Diag}
\Diag(\Gamma)\bydef\{ \varphi\in \Stab(\Gamma)\mid \forall g\in G\ \exists \mu_g\in\FF^\times
\text{ such that }\varphi\vert_{\cA_g}=\mu_g\id\}.
\end{equation}
One obtains easily the following equalities (see \cite[(1.6),(1.7)]{EKmon}):
\begin{equation}\label{eq:StabAutG_CentNorm}
\Stab(\Gamma)=\Cent_{\Aut(\cA)}\bigl(\Diag(\Gamma)\bigr),\quad
\Aut(\Gamma)=\Norm_{\Aut(\cA)}\bigl(\Diag(\Gamma)\bigr).
\end{equation}
That is the stabilizer (respectively, automorphism group) of $\Gamma$ is the centralizer
(resp., normalizer) in $\Aut(\cA)$ of the diagonal group.

\medskip

For arbitrary $\FF$, and keeping our hypothesis of finite dimensionality of $\cA$, the group of automorphisms $\Aut(\cA)$ must be substituted by the
\emph{affine group scheme of automorphisms} $\Auts(\cA)$.

The functorial point of view of affine group schemes, as considered in \cite{Waterhouse}, 
\cite[Chapter 6]{KMRT}, or \cite{Milne}, will be followed. The reader is referred to these references for
their basic facts.

Therefore, an affine group scheme is a representable functor from the category $\AlgF$ of unital
commutative associative algebras over $\FF$ to the category $\Grp$ of groups. In particular,
$\Auts(\cA)$ acts as follows on objects:
\[
\Auts(\cA)\colon R\mapsto \Aut_{\AlgR}(\cA\otimes_\FF R)
\]
(group of automorphisms of $\cA\otimes_\FF R$ as an algebra over $R$). The action on homomorphisms is
the natural one by `scalar extension'.

Given our grading $\Gamma$ in \eqref{eq:Gamma}, there is an induced homomorphism of affine
group schemes extending equation \eqref{eq:rhoG} (see \cite[\S1.4]{EKmon}):
\begin{equation}\label{eq:etaG}
\eta_\Gamma\colon \Ds(G)\longrightarrow \Auts(\cA),
\end{equation}
where $\Ds(G)$ is the diagonalizable group scheme represented by the group algebra $\FF G$. 
Hence, for any $R$ in $\AlgF$, the group $\Ds(G)(R)$ consists of the algebra homomorphisms
$\FF G\rightarrow R$, which can be identified with the group homomorphisms $G\rightarrow R^\times$, 
that is, with the characters of $G$ with values in $R^\times$. The homomorphism $\eta_\Gamma$
works as before: given a character $\chi\colon G\rightarrow R^\times$, the associated automorphism of
$\cA\otimes_\FF R$ is given by the same formula as in \eqref{eq:tauchi}: 
\[
\tau_\chi(x)=\chi(g)x
\]
for all $x\in (\cA\otimes_\FF R)_g\bydef \cA_g\otimes_\FF R$.

In this new setting, the grading $\Gamma$ is recovered again from $\eta_\Gamma$, and actually in a very
easy way. Take $R=\FF G$, and let $\chi\in \Ds(G)(\FF G)$ be the \emph{generic} element, that is,
the identity homomorphism $\FF G\rightarrow \FF G$, identified with the group
homomorphism $G\rightarrow (\FF G)^\times$ that takes any $g\in G$ to the element $g\in\FF G$.
Then for any $x\in\cA_g$, $\tau_\chi(x\otimes 1)=x\otimes g$, and $\cA_g$ consists of the 
`eigenvectors' of `eigenvalue'  $g$ (computed inside  $\cA\otimes_\FF \FF G$) of $\tau_\chi$.

The image of $\eta_\Gamma$ is now contained in the \emph{diagonal group scheme} of $\Gamma$
(see \cite[p.~23]{EKmon}), which the reader should compare 
with \eqref{eq:Diag}):
\[
\Diags(\Gamma)\colon R\longmapsto \{\varphi\in \Aut_{\AlgR}(\cA\otimes_\FF R)\mid 
\forall g\in G\ \varphi\vert_{\cA_g\otimes_\FF R}\in R^\times\id\}.
\]

We are now ready to explain what this note is about.
\bigskip

\emph{The goal is to consider the affine group scheme versions, $\Stabs(\Gamma)$ and
$\Auts(\Gamma)$, of the stabilizer and the automorphism group of $\Gamma$. The former is 
straightforward, but the latter has some subtleties.}

\emph{Once this is done, the scheme theoretic version
of Equation \eqref{eq:StabAutG_CentNorm} will be proven (Theorem \ref{th:CentNorm}). Again, the proof
for $\Auts(\Gamma)$ requires extra care.}

\bigskip

Recall that if $\Hs$ is a closed subgroup of $\Gs$, the centralizer and normalizer of $\Hs$ in $\Gs$
are the subschemes of $\Gs$ that act as follows on objects:
\[
\begin{split}
\Cents_\Gs(\Hs)\colon R&\longmapsto\{g\in\Gs(R)\mid \text{$\forall R$-algebra $R'$ and }
\forall h\in \Gs(R')\ ghg^{-1}=h\},\\
\Norms_\Gs(\Hs)\colon R&\longmapsto\{g\in\Gs(R)\mid \text{$\forall R$-algebra $R'$}\ 
g\Hs(R')g^{-1}=\Hs(R')\}.
\end{split}
\]
Here there is an abuse of notation, we denote by $g$ both an element in $\Gs(R)$ and its image
in $\Gs(R')$ under $\Gs(\iota)$, for $\iota\colon R\rightarrow R'$ the algebra homomorphism determining
the structure of $R'$ as an $R$-algebra. This should create no confusion.

\smallskip

Once the goal above is achieved in Section \ref{se:AutGamma}, it is natural to consider the quotient 
\[
\Auts(\Gamma)/\Stabs(\Gamma) 
=\Norms_{\Auts(\cA)}\bigl(\Diags(\Gamma)\bigr)/\Cents_{\Auts(\cA)}\bigl(\Diags(\Gamma)\bigr).
\]
In Section \ref{se:Weyl}, this scheme theoretic quotient  will be shown to be a constant group scheme 
$\Wsf(\Gamma)$  (Proposition \ref{pr:Weyl}), called the \emph{Weyl group scheme} of $\Gamma$; 
but its associated ordinary group (the group of rational points over $\FF$)
 is not, in general, our previous Weyl group $W(\Gamma)$, as shown in Example \ref{ex:cubic2}, but
rather the Weyl group $W(\Gamma_{\overline{\FF}})$ of the grading $\Gamma_{\overline{\FF}}$
obtained on $\cA\otimes_\FF\overline{\FF}$ by extension of scalars, where $\overline{\FF}$ denotes
an algebraic closure of $\FF$.

\bigskip

In what follows $\cA$ will always denote a finite-dimensional nonassociative algebra
 over an arbitrary ground field $\FF$. Unadorned tensor products will be 
assumed to be over $\FF$.


\bigskip

\section{The automorphism group scheme of a grading}\label{se:AutGamma}

In order to define the automorphism group scheme of a grading we need a small detour to deal
with gradings on vector spaces.

Given a finite-dimensional vector space $\cV$ over a field $\FF$ and a decomposition 
$\Gamma:\cV=\bigoplus_{s\in S}\cV_s$ as a
direct sum of subspaces, denote by $\GL(\cV,\Gamma)$ the subgroup of the general linear group
$\GL(\cV)$ consisting of those linear automorphisms $\varphi$ of $\cV$ such that, for any 
$s\in \supp(\Gamma)\bydef\{s\in S\mid \cV_s\neq 0\}$,  there
is an element $\sigma(s)\in \supp(\Gamma)$ such that $\varphi(\cV_s)=\cV_{\sigma(s)}$. Under these
circumstances, $\sigma$ is a permutation of $\supp(\Gamma)$ 
($\sigma\in\Sym(\supp(\Gamma))$) with the property that 
$\dim_\FF \cV_s=\dim_\FF \cV_{\sigma(s)}$ for all $s\in \supp(\Gamma)$.

The affine group scheme version $ \GLs(\cV,\Gamma)$ works as follows. Any algebra
$R$ in $\AlgF$ without proper idempotents is sent to
\begin{multline*}
 \GLs(\cV,\Gamma)(R)\bydef\{\varphi\in \Aut_{R-\Mod}(\cV\otimes R)\mid \\ 
\exists\sigma\in\Sym(\supp(\Gamma))\ \text{such that}\ 
\varphi(\cV_s\otimes R)=\cV_{\sigma(s)}\otimes R\ \forall s\in \supp(\Gamma)\}.
\end{multline*}

For an arbitrary algebra $R$ in $\AlgF$ this must be modified in order to get an affine group scheme:
\begin{equation}\label{eq:GLsVGamma}
\begin{split}
 \GLs(\cV,\Gamma)(R)&\bydef \{\varphi\in\Aut_{R\textrm{-Mod}}(\cV\otimes R)\mid 
        \exists n\in\NN,\, \exists e_1,\ldots,e_n\ \text{nonzero}\\
&\qquad  \ \text{orthogonal idempotents in $R$ with}\ 1=e_1+\cdots+e_n, \\ 
 &\qquad          \ 
         \text{and }\exists \sigma_1,\ldots,\sigma_n\in\Sym(\supp(\Gamma)),\ \text{such that}\\ 
&\qquad  
 \varphi(\cV_s\otimes Re_i)= \cV_{\sigma_i(s)}\otimes Re_i\ \forall s\in \supp(\Gamma),\,\forall i=1,\ldots,n\}.
\end{split}
\end{equation}
That is, $ \GLs(\cV,\Gamma)(R)$ consists of those $R$-linear automorphisms of $\cV\otimes R$ for which 
there are nonzero orthogonal idempotents $e_1,\ldots,e_n$, with $1=e_1+\cdots + e_n$, such that
the restriction of $\varphi$ to each $\cV\otimes Re_i$ permutes the nonzero components 
$\cV_s\otimes Re_i$. 
Note that 
the permutation may be different for each $i=1,\ldots,n$.

For any homomorphism $f\colon R\rightarrow S$ in $\AlgF$, the definition of $ \GLs(\cV,\Gamma)(f)$ 
is just given by
extension of scalars from $R$ to $S$ via $f$.

\smallskip

It is not difficult to check that $ \GLs(\cV,\Gamma)$ is indeed an affine group scheme (i.e., representable).
For this, let $H$ be the subgroup of $\Sym(\supp(\Gamma))$ consisting of those permutations $\sigma$ with
$\dim_\FF \cV_s=\dim_\FF \cV_{\sigma(s)}$ for any $s\in \supp(\Gamma)$. Write 
$\supp(\Gamma)=\{s_1,\ldots,s_m\}$ and take bases
$(v_1,\ldots,v_{n_1})$ of $\cV_{s_1}$, $(v_{n_1+1},\ldots,v_{n_1+n_2})$ of $\cV_{s_2}$, ...., 
$(v_{n_1+\cdots+n_{m-1}+1},\ldots,v_{n_1+\cdots +n_m})$ of $\cV_{s_m}$. Write 
$I_1=\{1,\ldots,n_1\}$, $I_2=\{n_1+1,\ldots,n_1+n_2\}$, ..., 
$I_m=\{n_1+\cdots+n_{m-1}+1,\ldots n_1+\cdots +n_m\}$. Then $ \GLs(\cV,\Gamma)$
is represented by the algebra
\[
\cA=\prod_{\sigma\in H}\cA_\sigma,
\]
 where, for each $\sigma\in H$, the algebra $\cA_\sigma$ is the quotient of the polynomial algebra
\[
\FF[X_{ij},Y\mid 1\leq i,j\leq n_1+\cdots+n_m]
\]
by the ideal generated by $\det(X_{ij})Y-1$ and by the $X_{ij}$'s, with  
$i\in I_p$, $j\in I_q$ and $s_q\neq \sigma(s_p)$:
\[
\cA_\sigma=\FF[X_{ij},Y]/\left(\det(X_{ij})Y-1,\, X_{ij}\mid i\in I_p,\, j\in I_q,\, 
s_q\neq \sigma(s_p)\right).
\]

Note that for $\sigma=\id$, the algebra $\cA_{\id}$ represents the \emph{stabilizer}  $\Stabs(\Gamma)$,
that takes any $R$ to the $R$-linear automorphisms of $\cV\otimes R$ fixing each component
$\cV_s\otimes R$.

\medskip

We are now in position to define the automorphism group scheme of a grading.

Let $\Gamma:\cA=\bigoplus_{g\in G}\cA_g$ be a grading by an abelian group $G$ of our algebra $\cA$. 
As before, we may consider the \emph{stabilizer affine group scheme} $\Stabs(\Gamma)$, that takes
any $R$ in $\AlgF$ to the group
\[
\Stabs(\Gamma)(R)\bydef\{\varphi\in\Aut_{\AlgR}(\cA\otimes R)\mid \varphi(\cA_g\otimes R)=\cA_g\otimes R\ \forall g\in G\},
\]
and acts naturally on homomorphisms in $\AlgF$. The \emph{automorphism group scheme of the
grading $\Gamma$}, denoted by $\Auts(\Gamma)$, is defined simply as the intersection
\begin{equation}\label{eq:AutsGamma}
\Auts(\Gamma)\bydef \Auts(\cA)\cap \GLs(\cA,\Gamma).
\end{equation}
Note that $\Auts(\Gamma)(\FF)$ is the traditional group $\Aut(\Gamma)$ of automorphisms of $\Gamma$.

\medskip

The next theorem is the main result of this note. It shows that the situation for algebraically closed 
fields of characteristic $0$ extends to arbitrary fields if affine group schemes are used.

\begin{theorem}\label{th:CentNorm}
Let $\Gamma:\cA=\bigoplus_{g\in G}\cA_g$ be a grading by an abelian group on a  finite-dimensional
algebra $\cA$ over an arbitrary field $\FF$. Then,
\begin{enumerate}
\item $\Stabs(\Gamma)$ is the centralizer in $\Auts(\cA)$ of the diagonal group scheme 
$\Diags(\Gamma)$:
\[
\Stabs(\Gamma)=\Cents_{\Auts(\cA)}\bigl(\Diags(\Gamma)\bigr).
\]

\item $\Auts(\Gamma)$ is the normalizer in $\Auts(\cA)$ of the diagonal group scheme
$\Diags(\Gamma)$:
\[
\Auts(\Gamma)=\Norms_{\Auts(\cA)}\bigl(\Diags(\Gamma)\bigr).
\]
\end{enumerate}
\end{theorem}

\begin{proof}
For the first part, $\Stabs(\Gamma)$ is clearly a subgroup scheme of the centralizer
$\Cents_{\Auts(\cA)}\bigl(\Diags(\Gamma)\bigr)$. 

For the converse, let $R$ be in $\AlgF$ and let $\varphi$ be an element of the centralizer
$\Cents_{\Auts(\cA)}\bigl(\Diags(\Gamma)\bigr)(R)$. Let $R'=RG$ be the group algebra of $G$ over $R$
and consider the `generic' diagonal automorphism $\psi\in \Diags(\Gamma)(R')$ given by
\[
\psi(a_g\otimes 1)=a_g\otimes g,
\]
for any $g\in G$ and $a_g\in\cA_g$. As a consequence, $\cA_g\otimes R'$ is the
eigenspace of $\psi$ of eigenvalue $g\in R'$. Then, extending $\varphi$ to $R'$, we get, for any
$g\in G$ and $a_g\in \cA_g$,
\[
\psi\varphi(a_g\otimes 1)=\varphi\psi(a_g\otimes 1)=\varphi(a_g\otimes g)=\varphi(a_g\otimes 1)g,
\]
and this shows that $\varphi(a_g\otimes 1)$ lies in $\cA_g\otimes R$, proving that $\varphi$ belongs
to $\Stabs(\Gamma)(R)$.

\smallskip

For the second part, write $\Ns=\Norms_{\Auts(\cA)}\bigl(\Diags(\Gamma)\bigr)$, and recall
that for any $R$ in $\AlgF$:
\[
\Ns(R)\bydef\{\varphi\in\Auts(\cA)(R)\mid \varphi\Diags(\Gamma)(R')\varphi^{-1}=\Diags(\Gamma)(R')\ 
\forall R'\text{ in}\,\AlgR\},
\]
while an element $\varphi\in\Auts(\cA)(R)$ lies in $\Auts(\Gamma)(R)$ if and only if there are nonzero
orthogonal idempotents $e_1,\ldots,e_n\in R$ with $1=e_1+\cdots +e_n$, and permutations
$\sigma_1,\ldots,\sigma_n$ of $\supp(\Gamma)$ such that
\[
\varphi(\cA_g\otimes Re_i)=\cA_{\sigma_i(g)}\otimes Re_i
\]
for all $i=1,\ldots,n$ and all $g\in\supp(\Gamma)$. For any $R'$ in $\AlgR$ and any 
$\delta\in\Diags(\Gamma)(R')$, $\delta$ acts on $\cA_{\sigma_i(g)}\otimes R'$ by multiplication
by a scalar $r'$, and
hence it acts on $\cA_{\sigma_i(g)}\otimes R'e_i$ by multiplication by $r'e_i$. As a consequence, 
$\varphi^{-1}\delta\varphi$ acts on $\cA_g\otimes R'e_i$ by multiplication by $r'e_i$. It follows then that
$\varphi^{-1}\delta\varphi$ acts diagonally, and similarly for $\varphi\delta\varphi^{-1}$. Hence,
$\varphi\Diags(\Gamma)(R')\varphi^{-1}=\Diags(\Gamma)(R')$ and $\varphi\in\Ns(R)$. We have
thus proved that $\Auts(\Gamma)$ is a subgroup scheme of $\Ns$.

Take now an automorphism  $\varphi\in\Ns(R)$. In order to check that $\varphi$ lies in 
$\Auts(\Gamma)(R)$
we will make some reductions:

\begin{enumerate}
\item We may assume that $R$ is finitely generated:

Indeed, fix a basis $\{v_1,\ldots,v_n\}$ of $\cA$ consisting of homogeneous elements. The matrices
of $\varphi$ and $\varphi^{-1}$ in the basis $\{v_i\otimes 1\mid 1\leq i\leq n\}$ of $\cA\otimes R$
involve only a finite number of elements of $R$. Let $R'$ be the subalgebra of $R$ generated by 
these elements. Let $\iota\colon R'\rightarrow R$ be the inclusion. Then, our $\varphi\in\Ns(R)$
is the image under $\Ns(\iota)$ of an element in $\Ns(R')$ and we can substitute $R$ by $R'$.

\smallskip

\item We may assume that $R$ has no proper idempotents:

Since we are assuming now that $R$ is finitely generated, take a family of nonzero orthogonal 
idempotents $e_1,\ldots,e_m$ with $m$ maximal (this number is bounded by the number of irreducible
components of the prime spectrum of $R$ because, for any $e=e^2$, the set of prime ideals 
containing $e$
is both open and closed). Then $1=e_1+\cdots+e_m$, as otherwise we may add 
the idempotent $1-(e_1+\cdots+e_m)$ to our family, and hence $R=Re_1\oplus \cdots\oplus Re_m$ and
each $Re_i$ has no proper idempotent. Now it is enough to note that
 $\Ns(R)=\Ns(Re_1)\times\cdots\times \Ns(Re_m)$.
\end{enumerate}

Hence assume that $\varphi$ is an element of $\Ns(R)$ with $R$ containing no proper idempotents.
As above, consider the
$\FF$-algebra $R'=RG$ (group algebra over $R$), and the diagonal element $\psi\in\Diags(\Gamma)(R')$
given by $\psi(a_g\otimes 1)=a_g\otimes g$ for any $g\in G$ and $a_g\in\cA_g$. Since $\varphi$ is in
$\Ns(R)$ we have $\varphi^{-1}\psi\varphi\in\Diags(\Gamma)(R')$, so for any $g\in \supp(\Gamma)$
there is an invertible scalar $s_g\in R'$ such that 
\[
\varphi^{-1}\psi\varphi\vert_{\cA_g\otimes R'}=s_g\id.
\]
Thus, $s_g=r_1h_1+\cdots +r_nh_n$ for some $h_1,\ldots,h_n\in G$ and nonzero scalars $r_i\in R$.

For $a\in \cA_g$ we want to prove that $\varphi(a\otimes 1)$ is homogeneous and that its degree
 depends only on $g$. We have
\begin{equation}\label{eq:1}
\psi\varphi(a\otimes 1)=\varphi\bigl(\varphi^{-1}\psi\varphi\bigr)(a\otimes 1)=\varphi(a\otimes s_g)=
\varphi(a\otimes 1)s_g.
\end{equation}
Write $\varphi(a\otimes 1)=\sum_{h\in G}a_h$ (finite sum), with $a_h\in\cA_h\otimes R$. Then also
\[
\psi\varphi(a\otimes 1)=\sum_{h\in G}\psi(a_h)=\sum_{h\in G}a_hh.
\]
The last equation, together with \eqref{eq:1}, gives
\begin{equation}\label{eq:2}
\sum_{h\in G}a_hh=\sum_{h\in G}a_hs_g\ \, (\in\cA\otimes R'),
\end{equation}
and since $\cA\otimes R'=\bigoplus_{g\in G}\bigoplus_{h\in G}\cA_g\otimes Rh$ we conclude that
\begin{equation}\label{eq:3}
\begin{aligned}
&a_h=0\text{\quad for }h\neq h_1,\ldots,h_n,\\
&a_{h_i}=a_{h_i}r_i\text{\quad for $i=1,\ldots,n$},\\
&a_{h_i}r_j=0\text{\quad for $i\neq j$}.
\end{aligned}
\end{equation}

Let $P$ be a prime ideal of $R$. For any $b\in \cA\otimes R$ denote by $\bar b$ its image in
$\cA\otimes R/P$, and for any $r\in R$ denote by $\bar r$ its image in the integral domain $R/P$.
As $s_g=r_1h_1+\cdots+r_nh_n$ is invertible in $R'=RG$, it follows that $r_1+\cdots +r_n$ is
invertible in $R$, because of the counit homomorphism $RG\rightarrow R$ taking any $g\in G$ to $1$.
Therefore $\bar r_1+\cdots+\bar r_n$ is invertible in $R/P$.

Also, since $\varphi$ is an automorphism, the induced map 
$\bar\varphi\in\Aut_{\Alg_{R/P}}(\cA\otimes R/P)$ is also an automorphism, so 
$0\neq \bar\varphi(a\otimes 1)=\bar a_{h_1}+\cdots+\bar a_{h_n}$. In particular, there is
an index $i$ with $\bar a_{h_i}\neq 0$. From \eqref{eq:3} we get $\bar a_{h_i}\bar r_i=\bar a_{h_i}$,
so $\bar r_i\neq 0$, and since $R/P$ is an integral domain, its action on the free $R/P$-module
$\cA\otimes R/P$ has no torsion, so $\bar r_j=0$, because $\bar a_{h_i}\bar r_j=0$ by \eqref{eq:3}.

Hence, for any prime ideal $P$ of $R$ there is a unique index $i$, $1\leq i\leq n$, such that 
$\bar r_i\neq 0$,
that is, $r_i\not\in P$. We thus get a locally constant map
\[
\Phi\colon\operatorname{Spec}(R)\longrightarrow \{1,\ldots,n\},
\]
where we consider the discrete topology on the right, with 
\[
\Phi^{-1}(i)=\{P\in\operatorname{Spec}(R)\mid r_i\not\in P\}
\]
which is the basic open set $D(r_i)$ in $\operatorname{Spec}(R)$. 

But $\operatorname{Spec}(R)$ is connected because $R$ does not contain proper idempotents, so we 
conclude that there is an index $i$ which, after relabeling, we may assume to be $1$, such that
for any prime ideal $P$, $r_1\not\in P$ while $r_2,\ldots,r_m\in P$. It follows that $r_1$ is a unit in $R$,
as it is not contained in any maximal ideal, while $r_2,\ldots,r_m$ are nilpotent. Now from \eqref{eq:3}
we get $a_{h_j}r_1=0$ for $j\neq 1$, and hence $a_{h_j}=0$ for $j\neq 1$ and the element $s_g$ in
\eqref{eq:1} is just $r_1h_1$. We conclude that $\varphi(a\otimes 1)$ belongs to $\cA_{h_1}\otimes R$,
and $h_1$ depends only on $g\in\supp(\Gamma)$ and not on the element $a\in\cA_g$.

Therefore, we have shown that for any $g\in\supp(\Gamma)$, there is an element $h\in\supp(\Gamma)$
with $\varphi(\cA_g\otimes R)\subseteq \cA_h\otimes R$, and the same works for $\varphi^{-1}$. In
conclusion, $\varphi$ permutes the homogeneous components, so $\varphi\in\Auts(\Gamma)(R)$, as
required.
\end{proof}

A remark is in order.

\begin{remark}\label{re:GD}
Let $\Gamma:\cA=\bigoplus_{g\in G}\cA_g$ be a grading by an abelian group on a  finite-dimensional
algebra $\cA$ over an arbitrary field $\FF$ as in Theorem \ref{th:CentNorm}, and let 
$\eta_\Gamma\colon\Ds(G)\rightarrow \Auts(\cA)$ be the associated morphism of group schemes in
\eqref{eq:etaG}. If $H$ is
the subgroup of $G$ generated by $\supp(\Gamma)$, then $\eta_\Gamma$ factors naturally as
\[
\begin{tikzcd}
\Ds(G) \arrow[r, twoheadrightarrow] &\Ds(H) \arrow[r, hookrightarrow] &\Auts(\cA),
\end{tikzcd}
\]
where the first morphism is the restriction to $H$ (a quotient map) and the second gives a closed
 imbedding. We can then identify $\Ds(H)$ with a subgroup of $\Auts(\cA)$. It is the image of the
morphism $\eta_\Gamma$.

The proof of the first part of Theorem \ref{th:CentNorm} works substituting $\Diags(\Gamma)$ by
$\Ds(H)$, and hence it gives
\[
\Stabs(\Gamma)=\Cents_{\Auts(\cA)}\bigl(\Ds(H)\bigr).
\]

For the second part, the proof of the converse works too, obtaining
\[
\Auts(\Gamma)\geq \Norms_{\Auts(\cA)}\bigl((\Ds(H)\bigr),
\]
but the equality fails in general, as Example \ref{ex:norm} shows.
\end{remark}

\begin{example}\label{ex:norm}
Let $\cA=\FF u\oplus\FF v$ be a two-dimensional algebra with trivial multiplication: $u^2=uv=vu=v^2=0$,
graded by the cyclic group $G=\langle g\rangle$ of order $6$, 
with the elements $u$ and $v$ homogeneous of degrees 
$\deg(u)=g^2$ and $\deg(v)=g^3$. The group $G$ is generated by the support 
$\{g^2,g^3\}$ of this grading $\Gamma$. 

For any algebra $R$ in $\AlgF$, the elements of $\Ds(G)(R)$, 
considered as elements in $\Auts(\cA)(R)$
are the automorphisms 
\[
\tau_\chi\colon u\mapsto \chi(g^2)u,\ v\mapsto \chi(g^3)v
\]
for a character (i.e., group homomorphism) $\chi\colon G\rightarrow R^\times$. 

The automorphism $\varphi$
that swaps $u$ and $v$ clearly belongs to $\Auts(\Gamma)(\FF)=\Aut(\Gamma)$, but if
$\chi\colon G\rightarrow (\FF G)^\times$ is the character with $\chi(g)=g$, then there is no character
$\xi\colon G\rightarrow(\FF G)^\times$ with $\varphi\tau_\chi\varphi^{-1}=\tau_\xi$, because this would
force $\xi(g^2)=g^3$, $\xi(g^3)=g^2$, which is impossible, as we would have $\xi(g^5)=g^5$, and
$\xi=\chi$ because $g^5$ is a generator of $G$, a contradiction.

Therefore, $\varphi$ lies in $\Auts(\Gamma)(\FF)$ but not in 
$\Norms_{\Auts(\cA)}\bigl(\Ds(G)\bigr)(\FF)$.

The computation of $\Diags(\Gamma)$ and of $\Auts(\Gamma)$ is left as an exercise for the reader.
\end{example}

\smallskip

Let us finish this section with a useful remark.

\begin{remark}\label{re:Hsmooth}
Assume that the ground field $\FF$ is algebraically closed and that $\Hs$ is a smooth closed subgroup
scheme of the affine algebraic group scheme $\Gs$. Then 
\begin{equation}\label{eq:NormGH}
\Norms_{\Gs}(\Hs)(\FF)=\Norm_{\Gs(\FF)}\bigl(\Hs(\FF)\bigr). 
\end{equation}
In particular, over an algebraically
closed field of characteristic $0$, so that all affine group schemes are smooth, 
if $\Gamma:\cA=\bigoplus_{g\in G}\cA_g$ is a grading by an abelian group 
on a a finite-dimensional
algebra $\cA$, then we recover the result \cite[(1.6)]{EKmon} that $\Aut(\Gamma)=\Auts(\Gamma)(\FF)$ 
is 
$\Norm_{\Aut(\cA)}\bigl(\Diag(\Gamma)\bigr)=\Norms_{\Auts(\cA)(\FF)}\bigl(\Diags(\Gamma)(\FF)\bigr)$.

To prove \eqref{eq:NormGH}, let $\cH$ be the Hopf algebra representing $\Gs$, and let $\cI$ be the
Hopf ideal defining $\Hs$. As this is smooth, $\cI=\bigcap_{x\in\Hs(\FF)}\cM_x$, where $\cM_x$ is
the maximal ideal corresponding to $x$: $\cM_x=\ker x\colon \cH\rightarrow\FF$.

Let $\Delta$ be the comultiplication on $\cH$, $S$ its antipode, and use Sweedler's notation to write
$(\Delta\otimes\id)\Delta(a)=\sum a_{(1)}\otimes a_{(2)}\otimes a_{(3)}$.
Any $g\in\Gs(\FF)=\Hom_{\AlgF}(\cH,\FF)$ acts by conjugation on $\cH$:
\[
\iota_g(a)=\sum g(a_{(1)})gS(a_{(3)})a_{(2)},
\]
so for any $R$ in $\AlgF$ and $h\in\Gs(R)$:
\[
ghg^{-1}\colon a\mapsto \sum g(a_{(1)})h(a_{(2)})gS(a_{(3)})=h(\iota_g(a)).
\]
Hence $ghg^{-1}=h\iota_g$. Note that $\iota_{g_1g_2}=\iota_{g_2}\iota_{g_1}$ for any 
$g_1,g_2\in\Gs(\FF)$.

For $g\in\Gs(\FF)$, $\iota_g(\cI)=\cI$ if and only if  for all $R\in\AlgF$ and all $h\in\Gs(R)$ we have
that $h(\cI)=0$  if and only if
$ghg^{-1}(\cI)=0$, which holds if and only if $g\in\Norms_{\Gs}(\Hs)(\FF)$.

Note that $\Norms_{\Gs}(\Hs)(\FF)$ is contained in $\Norm_{\Gs(\FF)}\bigl(\Hs(\FF)\bigr)$ trivially.
On the other hand, for $g\in \Norm_{\Gs(\FF)}\bigl(\Hs(\FF)\bigr)$, $ghg^{-1}\in\Hs(\FF)$ for
any $h\in\Hs(\FF)$, and hence $h\bigl(\iota_g(\cI)\bigr)=0$ for any $h\in \Hs(\FF)$, so that
$\iota_g(\cI)$ is contained in $\cM_x$ for all $x\in\Hs(\FF)$, and we conclude that $\iota_g(\cI)$ is
contained in $\cI$. As $g^{-1}$ lies in $\Norm_{\Gs(\FF)}\bigl(\Hs(\FF)\bigr)$ too, we get also
$\iota_{g^{-1}}(\cI)\subseteq \cI$, so $\iota_g(\cI)=\cI$ and $g$ is in $\Norms_{\Gs}(\Hs)(\FF)$,
as required.
\end{remark}

\smallskip

If $\Hs$ is not smooth the equality $\Norms_{\Gs}(\Hs)(\FF)=\Norm_{\Gs(\FF)}\bigl(\Hs(\FF)\bigr)$ is
not valid in general. 

\bigskip

\section{Weyl groups}\label{se:Weyl}

This last section is devoted to define the Weyl group scheme of a grading as a constant group scheme.

Given a grading $\Gamma:\cA=\bigoplus_{g\in G}\cA_g$  by an abelian group on a a finite-dimensional
algebra $\cA$ over an arbitrary field $\FF$, the definition of $\Auts(\Gamma)$ gives at once a
natural morphism of affine group schemes
\begin{equation}\label{eq:AutGammaSymSupp}
\Auts(\Gamma)\rightarrow \Symsf\bigl(\supp(\Gamma)\bigr),
\end{equation}
with kernel $\Stabs(\Gamma)$, and
where the scheme on the right is the constant group scheme associated to the symmetric group
on $\supp(\Gamma)$. (Recall that given a finite group $G$, its associated constant group
scheme $\mathsf{G}$ assigns to any $R$ in $\AlgF$ without proper idempotents the group $G$.)

\begin{definition}\label{df:Weyl}
The image of the morphism in \eqref{eq:AutGammaSymSupp} is called the \emph{Weyl group scheme} of
the grading $\Gamma$. It will be denoted by $\Wsf(\Gamma)$.
\end{definition}

Since any subgroup of a constant affine group scheme is constant, $\Wsf(\Gamma)$ is a constant group
scheme, and there appears the short exact sequence
\begin{equation}\label{eq:ses}
\begin{tikzcd}
1 \arrow[r] &\Stabs(\Gamma) \arrow[r, hookrightarrow] & \Auts(\Gamma) \arrow[r, twoheadrightarrow] 
& \Wsf(\Gamma) \arrow[r] & 1.
\end{tikzcd}
\end{equation}

If $\overline{\FF}$ denotes an algebraic closure of $\FF$, since 
$\Auts(\Gamma) \twoheadrightarrow \Wsf(\Gamma)$
 is a quotient map, the induced group homomorphism
$\Auts(\Gamma)(\overline{\FF}) \twoheadrightarrow
\Wsf(\Gamma)(\overline{\FF})$
 is surjective (see, e.g., \cite[Theorem A.48]{EKmon}), and
we get the short exact sequence of groups
\[
\begin{tikzcd}
1 \arrow[r] &\Stabs(\Gamma)(\overline{\FF}) \arrow[r, hookrightarrow] 
& \Auts(\Gamma)(\overline{\FF}) \arrow[r, twoheadrightarrow] 
& \Wsf(\Gamma)(\overline{\FF}) \arrow[r] & 1.
\end{tikzcd}
\]
Denote by $\Gamma_{\overline{\FF}}$ the $G$-grading induced by $\Gamma$ on 
$\cA_{\overline{\FF}}=\cA\otimes\overline{\FF}$: 
$(\cA_{\overline{\FF}})_g=\cA_g\otimes\overline{\FF}$ for all $g\in G$. The last 
short exact sequence becomes
\[
\begin{tikzcd}
1 \arrow[r] &\Stab(\Gamma_{\overline{\FF}}) \arrow[r, hookrightarrow] 
& \Aut(\Gamma_{\overline{\FF}}) \arrow[r, twoheadrightarrow] 
& \Wsf(\Gamma)(\overline{\FF}) \arrow[r] & 1,
\end{tikzcd}
\]
and this shows that $\Wsf(\Gamma)(\overline{\FF})$ is the ordinary Weyl group 
$W(\Gamma_{\overline{\FF}})$. In other words:

\begin{proposition}\label{pr:Weyl}
The Weyl group scheme of $\Gamma$ is the constant group scheme associated to the 
Weyl group of $\Gamma_{\overline{\FF}}$.
\end{proposition}

\smallskip

In general, $\Wsf(\Gamma)$ is not the constant group scheme associated to the ordinary
Weyl group $W(\Gamma)$, as the following example shows.

\begin{example}\label{ex:cubic2}
Let $\cA$ be the associative $\QQ$-algebra $\cA=\QQ 1\oplus\QQ u\oplus\QQ u^2$, with multiplication determined by
$u^3=2$. That is, $\cA$ is, up to isomorphism, 
the field  $\QQ(\sqrt[3]{2})$. The direct sum above gives a  
grading $\Gamma$ on $\cA$ by the cyclic group $C_3=\langle g\rangle$ of order $3$, where
$\deg(1)=e$, $\deg(u)=g$ and $\deg(u^2)=g^2$. The automorphism group
of $\cA$ is trivial, and hence so are $\Stab(\Gamma)$,  $\Aut(\Gamma)$, and $W(\Gamma)$.

However, $\cA_{\overline{\QQ}}=\overline{\QQ}1\oplus\overline{\QQ}w\oplus\overline{\QQ}w^2$, with $w=\frac{1}{\sqrt[3]{2}}u$, that satisfies $w^3=1$. Hence $\cA_{\overline{\QQ}}$ is isomorphic to 
$\overline{\QQ}\times \overline{\QQ}\times \overline{\QQ}$. Here the automorphism determined
by $\varphi(w)=w^2$ belongs to $\Aut(\Gamma_{\overline{\QQ}})$. 
It follows that the Weyl group $W(\Gamma_{\overline{\QQ}})$ is the cyclic group of order $2$, and hence
$\Wsf(\Gamma)$ is the associated constant group scheme $\Csf_2$. 

Trivially, $\Stabs(\Gamma)$ is, up to isomorphism, the affine group scheme of third roots of unity
 $\bmu_3$, and the short exact sequence \eqref{eq:ses} becomes
\begin{equation}\label{eq:ses2}
1\longrightarrow \bmu_3\longrightarrow \Auts(\Gamma)\longrightarrow \Csf_2\longrightarrow 1.
\end{equation}
Since $\Auts(\Gamma)(\QQ)$ is trivial, it turns out that the quotient map 
$\Auts(\Gamma)\rightarrow\Csf_2$ is not surjective for rational points. It also follows that
the short exact sequence \eqref{eq:ses2} does not split, because the group 
$\Aut(\Gamma)=\Auts(\Gamma)(\QQ)$ is trivial, so it does not contain cyclic subgroups of order $2$.

Note that since $\cA$ is an \'etale algebra, $\Auts(\cA)$ is an \'etale group scheme (see e.g., 
\cite[Exercise VI.1]{KMRT}). Also,
$\Auts(\cA)(\overline{\QQ})=\Aut(\cA_{\overline{\QQ}})\simeq \Aut(\overline{\QQ}^3)$, which
is the symmetric group of degree $3$ (and order $6$). From here we get that $\Auts(\Gamma)$ is the whole $\Auts(\cA)$.
\end{example}

The last result is a consequence of Theorem \ref{th:CentNorm}. It deals with fine gradings, which are
those gradings $\Gamma:\cA=\bigoplus_{g\in G}\cA_g$ such that there exists no other grading
$\Gamma':\cA=\bigoplus_{h\in H}\cA'_h$ by an abelian group with the property that every 
nonzero homogeneous component $\cA'_h$ of $\Gamma'$ is 
contained in a homogeneous component $\cA_g$ of $\Gamma$, with at least one of these containments
being strict. In a way, fine gradings are those whose homogeneous components are small, and this
translates into the fact that its diagonal group scheme $\Diags(\Gamma)$ is a maximal
diagonalizable subgroup scheme (or quasitorus) of $\Auts(\cA)$. (See \cite[Proposition 1.37]{EKmon}.)

Given a grading $\Gamma:\cA=\bigoplus_{g\in G}\cA_g$, the diagonal group scheme $\Diags(\Gamma)$
is a diagonalizable group scheme, and hence it is isomorphic to $\Ds(U)$ for an abelian group $U$. 
This is the \emph{universal group} of $\Gamma$. The associated morphism 
$\eta_\Gamma\colon \Ds(G)\rightarrow \Auts(\cA)$ factors as 
$\Ds(G)\rightarrow \Ds(U)\simeq \Diags(\Gamma)\hookrightarrow \Auts(\cA)$, and this induces a group
homomorphism $U\rightarrow G$. The grading $\Gamma$ can be realized as a grading by $U$, and as
such, the homomorphism $U\rightarrow G$ takes bijectively the support of $\Gamma$ as an $U$-grading
with the support of $\Gamma$ as a $G$-grading. Therefore, when considered as a grading by $U$, 
the morphism $\eta_\Gamma\colon \Ds(U)\rightarrow \Auts(\cA)$ takes $\Ds(U)$ isomorphically onto  
$\Diags(\Gamma)$.

\begin{corollary}\label{co:fine}
Let $\cA$ and $\cB$ be two finite-dimensional algebras over a field $\FF$ and let 
$\phi\colon\Auts(\cA)\rightarrow \Auts(\cB)$ be an isomorphism. 
Let $\Gamma$ be a fine grading on $\cA$ with universal group $U$. Let $\eta_\Gamma\colon \Ds(U)\rightarrow \Auts(\cA)$ be the associated 
morphism, and let $\Gamma'$ be the grading by $U$ on $\cB$ associated to the morphism 
$\phi\eta_\Gamma$.
Then $\Gamma'$ is fine and $\phi$ restricts to isomorphisms $\Stabs(\Gamma)\simeq\Stabs(\Gamma')$
and $\Auts(\Gamma)\simeq\Auts(\Gamma')$. As a consequence, $\phi$ induces an isomorphism
$\Wsf(\Gamma)\simeq\Wsf(\Gamma')$.
\end{corollary}
\begin{proof}
Since $\Gamma$ is fine, $\Diags(\Gamma)$ is a maximal quasitorus of $\Auts(\cA)$. 
Note that $\Diags(\Gamma)$ is the image under $\eta_\Gamma$ of $\Ds(U)$. 
The image under $\phi$
of $\Diags(\Gamma)$ is then a maximal quasitorus of $\Auts(\cB)$. It follows that $\Gamma'$ is fine
with universal group $U$, and that the image of $\Diags(\Gamma)$ under $\phi$, that is, the
the image of $\Ds(U)$ under $\eta_\Gamma'=\phi\eta_\Gamma$, is
$\Diags(\Gamma')$. Now Theorem \ref{th:CentNorm} implies at once that $\phi$ restricts to isomorphisms $\Stabs(\Gamma)\simeq\Stabs(\Gamma')$
and $\Auts(\Gamma)\simeq\Auts(\Gamma')$ and, as a consequence, it induces an isomorphism
$\Wsf(\Gamma)\simeq\Wsf(\Gamma')$.
\end{proof}


\end{document}